\documentclass[english]{article}
\usepackage[margin=2.4cm]{geometry}
\usepackage{hyperref}
\usepackage{amsthm,amsfonts,amssymb,amsmath}
\usepackage{graphicx}

\newtheorem{theo}{Theorem}

\newtheorem{lem}[theo]{Lemma}

\newtheorem{coro}[theo]{Corollary}

\theoremstyle{definition}

\newtheorem{remark}[theo]{Remark}

\numberwithin{equation}{section}
\numberwithin{figure}{section}

\usepackage{babel}
\usepackage[T1]{fontenc}
\usepackage[utf8]{inputenc}

\newcommand{\su}{\subseteq}
\newcommand{\F}{\mathbb{F}}
\newcommand{\Fp}{\mathbb{F}_p}
\newcommand{\Fpn}{\mathbb{F}_p^{n}}
\newcommand{\Z}{\mathbb{Z}}
\newcommand{\s}{\mathfrak{s}}
\newcommand{\g}{\mathfrak{g}}
\newcommand{\rt}{r}

\begin{document}
\title{Erd\H{o}s-Ginzburg-Ziv constants by avoiding three-term arithmetic progressions}
\author{Jacob Fox\thanks{Department of Mathematics, Stanford University, Stanford, CA 94305. Email: {\tt jacobfox@stanford.edu}. Research supported by a Packard Fellowship and by NSF Career Award DMS-1352121.} \and Lisa Sauermann\thanks{Department of Mathematics, Stanford University, Stanford, CA 94305. Email: {\tt lsauerma@stanford.edu}. Research supported by Jacob Fox's Packard Fellowship.}}
\date{}

\maketitle

\begin{abstract}
For a finite abelian group $G$, the Erd\H{o}s-Ginzburg-Ziv constant $\mathfrak{s}(G)$ is the smallest $s$ such that every sequence of $s$ (not necessarily distinct) elements of $G$ has a zero-sum subsequence of length $\operatorname{exp}(G)$. For a prime $p$, let $r(\mathbb{F}_p^n)$ denote the size of the largest subset of $\mathbb{F}_p^n$ without a three-term arithmetic progression. Although similar methods have been used to study $\mathfrak{s}(G)$ and $r(\mathbb{F}_p^n)$, no direct connection between these quantities has previously been established. We give an upper bound for $\mathfrak{s}(G)$ in terms of $r(\mathbb{F}_p^n)$ for the prime divisors $p$ of $\operatorname{exp}(G)$.  For the special case $G=\mathbb{F}_p^n$, we prove $\mathfrak{s}(\mathbb{F}_p^n)\leq 2p\cdot r(\mathbb{F}_p^n)$. Using the upper bounds for $r(\mathbb{F}_p^n)$ of Ellenberg and Gijswijt, this result improves the previously best known upper bounds for $\mathfrak{s}(\mathbb{F}_p^n)$ given by Naslund.
\end{abstract}

\section{Introduction}
Let $G$ be a non-trivial finite abelian group, additively written. We denote the exponent of $G$ by $\exp(G)$; this is the least common multiple of the orders of all elements of $G$.

The Erd\H{o}s-Ginzburg-Ziv constant $\s(G)$ is the smallest integer $s$ such that every sequence of $s$ (not necessarily distinct) elements of $G$ has a subsequence of length $\exp(G)$ whose elements sum to zero in $G$. Furthermore, let $\g(G)$ denote the smallest integer $a$ such that every subset $A\su G$ of size $\vert A\vert\geq a$  contains $\exp(G)$ distinct elements summing to zero in $G$. It is easy to see that $\g(G)\leq \s(G)$ and $\s(G)\leq (\exp(G)-1)(\g(G)-1)+1$.

A three-term arithmetic progression is a subset of $G$ consisting of three distinct elements such that the sum of two of these elements equals twice the third element, i.e.\ a set of the form $\lbrace x,y,z\rbrace\su G$ with $x,y,z$ distinct and $x+z=2y$. For $y\in G$, a three-term arithmetic progression with middle term $y$ is a set of the form $\lbrace x,y,z\rbrace\su G$ with $x,y,z$ distinct and $x+z=2y$. For a finite abelian group $G$, let $\rt(G)$ denote the largest size of a subset of $G$ without a three-term arithmetic progression. Note that $\rt(\F_2^n)=2^n$, since there are no three-term arithmetic progressions in $\F_2^n$. Also note that in the case of $G=\F_3^n$, a three-term arithmetic progression is the same as a set of three distinct elements summing to zero, hence $\rt(\F_3^n)=\g(\F_3^n)-1$ (see also \cite{alondub2} and \cite{edeletal}).

In 1961, Erd\H{o}s, Ginzburg and Ziv \cite{egz} proved for each positive integer $k$ that any sequence of $2k-1$ integers contains a subsequence of length $k$ whose sum is divisible by $k$. The same statement is clearly not true for sequences of length $2k-2$. Thus, their result can be reformulated as $\s(\Z/k\Z)=2k-1$. The work of Erd\H{o}s, Ginzburg and Ziv \cite{egz} was the starting point for a whole field studying different zero-sum problems in various finite abelian groups; see for example the survey article by Gao and Geroldinger \cite{gaogerold}.

Note that $\s((\Z/k\Z)^n)$ has a simple geometric interpretation: it is the smallest number $s$ such that among any $s$ points in the lattice $\Z^n$ one can choose $k$ points such that their centroid is again a lattice point in $\Z^n$. Harborth \cite{harb} investigated $\s((\Z/k\Z)^n)$ in this context and was the first to study Erd\H{o}s-Ginzburg-Ziv constants for non-cyclic groups. He proved
$$(k-1)2^{n}+1\leq \s((\Z/k\Z)^n)\leq (k-1)k^n+1,$$
where the upper bound is easily obtained from the pigeonhole principle. Harborth  \cite{harb} also established $\s((\Z/2^m\Z)^n)=(2^m-1)2^{n}+1$ and in particular $\s(\F_2^n)=2^{n}+1$. For $n=2$, Reiher \cite{reiher} determined that $\s((\Z/k\Z)^2)=4k-3$ for all positive integers $k$. Alon and Dubiner \cite{alondub} proved $\s((\Z/k\Z)^n)\leq (cn\log n)^nk$ for some absolute constant $c$. Hence, for any fixed $n$, the quantity $\s((\Z/k\Z)^n)$ grows linearly with $k$. It remains an interesting question to estimate $\s((\Z/k\Z)^n)$ when $k$ is fixed and $n$ is large. Elsholtz \cite{elsholtz} obtained the lower bounds $\s((\Z/k\Z)^n)\geq 1.125^{\lfloor n/3\rfloor}(k-1)2^{n}+1$ for $k\geq 3$ odd and all $n$, and in particular $\s((\Z/k\Z)^n)\geq 2.08^n$ if $k\geq 3$ is odd and $n$ is sufficiently large.

For general finite abelian groups, Gao and Yang \cite{gaoyang} proved the upper bound $\s(G)\leq \vert G\vert+\exp(G)-1$ (see also \cite[Theorem 5.7.4]{bookgeroldhaltkoch}). 
Alon and Dubiner's result \cite{alondub} has been used to obtain upper bounds on $\s(G)$ 
when $G$ has small rank (the rank of $G$ is $\max(n_1,\dots,n_m)$, where $n_1,\dots,n_m$ are defined as in Theorem \ref{thm1} below), see \cite[Theorem 1.4]{edeletal} and \cite[Theorem 1.5]{chintamanietal}. In this paper, we will focus on the opposite case where at least one of $n_1,\dots,n_m$ is large compared to $\exp(G)$.

The case $G=\Fpn$ for a prime $p\geq 3$ has attracted particular interest. In this case, Naslund \cite{naslund} proved that $\g(\Fpn)\leq (2^p-p-2)\cdot (J(p)p)^n$ and  $\s(\Fpn)\leq (p-1)2^p\cdot (J(p)p)^n$, where $0.8414\leq J(p)\leq 0.9184$. To prove these bounds, Naslund introduced a variant of Tao's slice rank method \cite{tao}. Tao developed this method as an alternative formulation of the proof of $\rt(\Fpn)\leq (J(p)p)^n$ by Ellenberg and Gijswijt \cite{ellengijs}, which in turn used the new polynomial method introduced by Croot, Lev and Pach \cite{crootlevpach} to prove $\rt((\Z/4\Z)^n)\leq 3.62^n$. Note that the constant $J(p)p$ in Naslund's bounds for $\g(\Fpn)$ and $\s(\Fpn)$ is the same as in the bound $\rt(\Fpn)\leq (J(p)p)^n$ by Ellenberg and Gijswijt \cite{ellengijs}, see also \cite{blasiaketal}.

While similar methods have been applied to prove upper bounds for the Erd\H{o}s-Ginzburg-Ziv constant and upper bounds for sets without arithmetic progressions,  no direct connection between the two problems has previously been established (apart from the case $G=\F_3^n$ mentioned above). In this note, we derive upper bounds for $\s(G)$ for all finite abelian groups $G$ in terms of $\rt(\Fpn)$ for the prime divisors $p$ of $\exp(G)$. It is also possible to prove an upper bound of the form $\s(G)\leq O(\exp(G)\rt(G))$. However, $\exp(G)\rt(G)$ is usually much larger than our upper bound in Theorem \ref{thm1}.

\begin{theo}\label{thm1} Let $G$ be a non-trivial finite abelian group. Let $p_1,\dots,p_m$ be the distinct prime factors of $\exp(G)$. When writing $G$ as a product of cyclic groups of prime power order, all the occurring prime powers are powers of $p_1,\dots,p_m$. For $i=1,\dots,m$, let $n_i$ be the number of cyclic factors of $G$ whose order is a power of $p_i$. Then we have
$$\s(G)< 3\exp(G)\cdot (\rt(\F_{p_1}^{n_1})+\dots+\rt(\F_{p_m}^{n_m})).$$
\end{theo}

For the case $G=(\Z/k\Z)^n$ we obtain the following corollary (note that $\Z/k\Z$ has precisely one cyclic factor of prime power order for each distinct prime dividing $k$, hence $(\Z/k\Z)^n$ has precisely $n$ cyclic factors for each distinct prime dividing $k$).

\begin{coro}\label{coro2} Let $k\geq 2$ be an integer and let $p_1,\dots,p_m$ be its distinct prime factors. Then we have
$$\s((\Z/k\Z)^n)< 3k(\rt(\F_{p_1}^{n})+\dots+\rt(\F_{p_m}^{n}))$$
for every positive integer $n$.
\end{coro}

Recall that $\rt(\F_2^n)=2^n$. For primes $p\geq 3$ it is known from \cite{ellengijs} and \cite{blasiaketal} that $\rt(\Fpn)\leq (J(p)p)^n$, with $0.8414\leq J(p)\leq 0.9184$ and with $J(p)$ being a decreasing function that tends to $0.8414...$ as $p\to\infty$ (see \cite{blasiaketal} for more details and for the precise definition of the function $J(p)$). As a lower bound, we have $\rt(\Fp)\geq p^{1-o(1)}$ by Behrend's construction \cite{behrend} and $\rt(\Fpn)\geq p^{(1-o(1))n}$ by taking a product with Behrend's construction in each coordinate (here $o(1)\to 0$ as $p\to\infty$ independently of $n$). Furthermore, Alon, Shpilka and Umans \cite{alonshpilkaumans}, relying on a construction of Salem and Spencer \cite{salemspencer}, proved $\rt(\Fpn)\geq (p/2)^{(1-o(1))n}$, where $o(1)\to 0$ as $n\to\infty$ with $p$ fixed. A variant of Behrend's construction due to Alon gives an improvement of the $o(1)$-term (see \cite[Lemma 17]{foxpham}). Note that in light of $\rt(\Fpn)\geq (p/2)^{(1-o(1))n}$, for large $n$ and odd $k\geq 3$ there is still a big gap between Elsholtz' lower bound $\s((\Z/k\Z)^n)\geq 2.08^n$ and the upper bound for $s((\Z/k\Z)^n)$ in Corollary \ref{coro2}.

The bounds in Theorem \ref{thm1} and Corollary \ref{coro2} look clean and simple, but they are not the optimal results that can be obtained from our arguments (see Remark \ref{complicatedremark} and the second inequality in Lemma \ref{lemma-product-pgroups} where certain terms are just ignored). However, the improvements when optimizing the estimates in our proof are not very significant as long as $\exp(G)$ is small compared to at least one of $n_1,\dots,n_m$.

In Section \ref{sect2}, we will first prove the following upper bounds for $\g(\Fpn)$ and $\s(\Fpn)$ using the probabilistic method. In Section \ref{sect3} we will then deduce Theorem \ref{thm1} from Theorem \ref{thm4}.

\begin{theo}\label{thm3} Let $p\geq 3$ be a prime and $n\geq 2$ be an integer. Then $\g(\Fpn)\leq 2p\cdot \rt(\Fp^{n-1})$.
\end{theo}

\begin{theo}\label{thm4} Let $p\geq 3$ be a prime and $n\geq 1$ be an integer. Then $\s(\Fpn)\leq 2p\cdot \rt(\Fpn)$.
\end{theo}

For $p\geq 3$ prime, using $\rt(\Fpn)\leq (J(p)p)^n$, we obtain
$$\g(\Fpn)\leq 2p\cdot \rt(\Fp^{n-1})\leq 2p\cdot (J(p)p)^{n-1}<3(J(p)p)^{n}$$
and
$$\s(\Fpn)\leq 2p\cdot \rt(\Fpn)\leq 2p\cdot (J(p)p)^{n},$$
which slightly improves the previously best known bounds for $\g(\Fpn)$ and $\s(\Fpn)$ from \cite{naslund}.

To obtain an upper bound for $\g(\Fpn)$ in terms of $\rt(\Fpn)$, note that a product construction shows
$$\rt(\Fpn)\geq \rt(\Fp^{n-1})\cdot \rt(\Fp)\geq 2\rt(\Fp^{n-1})p^{1-o(1)}.$$
Hence, Theorem \ref{thm3} implies $\g(\Fpn)\leq  p^{o(1)}\rt(\Fp^{n})$, where $o(1)\to 0$ as $p\to\infty$ independently of $n$.

\section{Proof of Theorems \ref{thm3} and \ref{thm4}}\label{sect2}

\begin{lem}\label{lem1} Let $p\geq 3$ be a prime and $n\geq 1$. If $A\su \Fpn$ does not contain $p$ distinct elements summing to zero, then for every $x\in A$ the set $A$ contains at most $\frac{p-3}{2}$ different three-term arithmetic progressions with middle term $x$.
\end{lem}

\begin{proof}
Suppose that for some $x\in A$ the set $A$ contains $\frac{p-1}{2}$ different three-term arithmetic progressions with middle term $x$. Each of them consists of $x$ and two more elements of $A$ whose sum equals $2x$. So we obtain $\frac{p-1}{2}$ pairs of elements of $A$, each pair with sum $2x$. It is not hard to see that the $p-1$ elements of $A$ involved in these $\frac{p-1}{2}$ pairs are all distinct and distinct from $x$. So taking these $p-1$ elements together with $x$ itself, we obtain $p$ distinct elements of $A$ with sum 
$\frac{p-1}{2}\cdot 2x+x=p\cdot x=0$.
This is a contradiction to the assumption on $A$.
\end{proof}

\begin{remark}\label{rem21}By definition, $\rt(\Fp^{n-1})$ is the largest size of a subset of $\Fp^{n-1}$ without a three-term arithmetic progression. Let $V$ be an affine subspace of dimension $n-1$ in $\Fpn$, i.e.\ a hyperplane in $\Fpn$. We can consider a translation moving $V$ to the origin (so that it becomes a linear subspace of dimension $n-1$) and then an isomorphism to $\Fp^{n-1}$. This gives a bijection between $V$ and $\Fp^{n-1}$ which preserves three-term arithmetic progressions. Hence the largest size of a subset of $V$ without a three-term arithmetic progression is also equal to $\rt(\Fp^{n-1})$.
\end{remark}

We will now prove Theorem \ref{thm3}. Note that $\exp(\Fpn)=p$.

\begin{proof}[Proof of Theorem \ref{thm3}]
Let $A\su \Fpn$ be a subset that does not contain $p$ distinct elements summing to zero. We need to show that $\vert A\vert < 2p\cdot \rt(\Fp^{n-1})$.

By Lemma \ref{lem1} we know that for every $x\in A$ the set $A$ contains at most $\frac{p-3}{2}$ different three-term arithmetic progressions with middle term $x$. Hence the total number of three-term arithmetic progressions contained in the set $A$ is at most $\frac{p-3}{2}\vert A\vert$.

Pick an affine subspace $V$ of dimension $n-1$ in $\Fpn$ uniformly at random. Let $X_1=\vert A\cap V\vert$ and let $X_2$ be the number of three-term arithmetic progressions that are contained in $A\cap V$. Since each point of $A$ is contained in $V$ with probability $\frac{1}{p}$, we have $\mathbb{E}[X_1]=\frac{1}{p}\vert A\vert$.

For any three-term arithmetic progression, the probability that its first element is contained in $V$ is equal to $\frac{1}{p}$. Conditioned on this, the probability that its second element is also contained in $V$ is $\frac{p^{n-1}-1}{p^n-1}< \frac{1}{p}$ (and note that then the third element will be contained in $V$ as well). Hence for any three-term arithmetic progression contained in $A$, the probability that it is contained in $A\cap V$ is less than $\frac{1}{p^{2}}$. Since $A$ contains at most $\frac{p-3}{2}\vert A\vert$ three-term arithmetic progressions, we obtain
$$\mathbb{E}[X_2]<\frac{1}{p^{2}}\cdot \frac{p-3}{2}\vert A\vert<\frac{1}{2p}\vert A\vert.$$
Thus, $\mathbb{E}[X_1-X_2]>\frac{1}{2p}\vert A\vert$. So we can choose an affine subspace $V$ of dimension $n-1$ in $\Fpn$ such that $X_1-X_2>\frac{1}{2p}\vert A\vert$. Let $B$ be a set obtained from $A\cap V$ after deleting one element from each three-term arithmetic progression contained in $A\cap V$. Then $\vert B\vert\geq X_1-X_2>\frac{1}{2p}\vert A\vert$. By construction, $B$ is a subset of $V$ that does not contain any three-term arithmetic progression. By Remark \ref{rem21}, we can conclude that $\vert B\vert\leq \rt(\Fp^{n-1})$. Thus, $\frac{1}{2p}\vert A\vert<\vert B\vert\leq \rt(\Fp^{n-1})$ and therefore $\vert A\vert < 2p\cdot \rt(\Fp^{n-1})$.
\end{proof}

Our proof of Theorem \ref{thm3} is somewhat similar to the first half of the proof of Proposition 2.5 in Alon's paper \cite{alon}. There, he also considered points which are the middle term of only few three-term arithmetic progressions and obtained a subset without any three-term arithmetic progressions, yielding a contradiction. However, Alon's work \cite{alon} is in a very different context and does not use a subspace sampling argument.

Finally, we will deduce Theorem \ref{thm4} from Theorem \ref{thm3}.

\begin{proof}[Proof of Theorem \ref{thm4}] Assume we are given a sequence of vectors in $\Fp^n$ without a zero-sum subsequence of length $p$. Every vector occurs at most $p-1$ times in the sequence. Hence by attaching one additional coordinate we can make all the vectors in the sequence distinct. This way, we obtain a subset of $\Fp^{n+1}$ without $p$ distinct elements summing to zero. Since this subset has size at most $\g(\Fp^{n+1})-1$, we can conclude that the original sequence had length at most $\g(\Fp^{n+1})-1$. This shows $\s(\Fp^n)\leq \g(\Fp^{n+1})$ and together with Theorem \ref{thm3} with $n$ replaced by $n+1$, we obtain $\s(\Fp^n)\leq \g(\Fp^{n+1})\leq 2p\cdot \rt(\Fp^n)$ as desired.
\end{proof}

\section{Proof of Theorem \ref{thm1}}\label{sect3}

In this section we will first bound $\s(G)$ for any finite abelian group $G$ by terms of the form $\s(\Fp^{n})$. Then, applying Theorem \ref{thm4}, we will obtain Theorem \ref{thm1}.

The following lemma was proved by Chi, Ding, Gao, Geroldinger and Schmid \cite[Proposition 3.1]{chietal} and is a generalization of \cite[Hilfssatz 2]{harb}. For the reader's convenience we repeat the proof here.

\begin{lem}[Proposition 3.1 in \cite{chietal}]\label{lemmaprod}Let $G$ be a non-trivial finite abelian group and $H\subseteq G$ be a subgroup such that $\exp(G)=\exp(H)\exp(G/H)$. Then
$$\s(G)\leq \exp(G/H)(\s(H)-1)+\s(G/H).$$
\end{lem}

\begin{proof}Consider a sequence of length $\exp(G/H)(\s(H)-1)+\s(G/H)$ with elements in $G$. Then we can find a subsequence of length $\exp(G/H)$ summing to zero in $G/H$, i.e. summing to an element of $H$. Delete this subsequence and repeat. We can do this $\s(H)$ many times (since after $\s(H)-1$ many times we still have $\s(G/H)$ elements left). So we find $\s(H)$ disjoint subsequences each of length $\exp(G/H)$ and the sum of each of the subsequences is in $H$. Now writing down these $\s(H)$ sums, we get a sequence of length $\s(H)$ with elements in $H$. So we can choose $\exp(H)$ of them summing to zero. Now taking the union of the corresponding subsequences of the original sequence we obtain $\exp(H)\exp(G/H)=\exp(G)$ elements summing to zero.\end{proof}

\begin{lem}\label{lemma-pgroup}For any finite abelian $p$-group $G=(\Z/p^{a_1}\Z)\times \dots\times (\Z/p^{a_n}\Z)$, where $a_1\geq \dots\geq a_n$ are positive integers and $p\geq 2$ is  prime, we have
$$\s(G)=\s((\Z/p^{a_1}\Z)\times \dots\times (\Z/p^{a_n}\Z))\leq \frac{p^{a_1}-1}{p-1}\s(\Fp^{n})< \frac{\exp(G)}{p-1}\s(\Fp^{n}).$$
\end{lem}

\begin{proof}Since $\exp(G)=p^{a_1}$, the second inequality is clearly true. Now, let us prove the first inequality by induction on $a_1$. If $a_1=1$, then $a_1= \dots= a_n=1$ and so 
$$\s((\Z/p^{a_1}\Z)\times \dots\times (\Z/p^{a_n}\Z))=\s(\Fp^{n})= \frac{p^{a_1}-1}{p-1}\s(\Fp^{n}).$$
For $a_1>1$ we can apply Lemma \ref{lemmaprod} to $H=pG$. Indeed, $G/H\cong \Fp^{n}$ and $H\cong (\Z/p^{a_1-1}\Z)\times \dots\times (\Z/p^{a_n-1}\Z)$. In particular, $\exp(G)=p^{a_1}=p^{a_1-1}\cdot p=\exp(H)\exp(G/H)$. So by Lemma \ref{lemmaprod} we have
$$\s(G)\leq \exp(\Fp^{n})(\s((\Z/p^{a_1-1}\Z)\times \dots\times (\Z/p^{a_n-1}\Z))-1)+\s(\Fp^{n})<p\s((\Z/p^{a_1-1}\Z)\times \dots\times (\Z/p^{a_n-1}\Z))+\s(\Fp^{n}).$$
Let $n'\leq n$ be such that $a_1\geq \dots\geq a_{n'}\geq 2$ and $a_{n'+1}=\dots=a_n=1$. Then by the induction assumption we have
$$\s((\Z/p^{a_1-1}\Z)\times \dots\times (\Z/p^{a_n-1}\Z))=\s((\Z/p^{a_1-1}\Z)\times \dots\times (\Z/p^{a_{n'}-1}\Z))\leq \frac{p^{a_1-1}-1}{p-1}\s(\Fp^{n'})\leq \frac{p^{a_1-1}-1}{p-1}\s(\Fp^{n}).$$
Thus,
$$\s(G)=\s((\Z/p^{a_1}\Z)\times \dots\times (\Z/p^{a_n}\Z))\leq p\cdot \frac{p^{a_1-1}-1}{p-1}\s(\Fp^{n})+\s(\Fp^{n})=\frac{p^{a_1}-1}{p-1}\s(\Fp^{n}),$$
completing the induction.\end{proof}

\begin{remark}\label{complicatedremark} The proof of Lemma \ref{lemma-pgroup} also gives the stronger but more complicated bound
$$\s((\Z/p^{a_1}\Z)\times \dots\times (\Z/p^{a_n}\Z))\leq \sum_{j=1}^{a_1}p^{j-1}\s(\Fp^{b_j}),$$
where $b_j=\max\,\lbrace i\mid a_i\geq j\rbrace$ for $j=1,\dots, a_1$. Note that $b_1\geq \dots \geq b_{a_1}$ is the conjugate of $a_1\geq \dots \geq a_n$ in the sense of Young diagrams.
\end{remark}

\begin{lem}\label{lemma-product-pgroups}Let $G$ be a non-trivial finite abelian group. Let $p_1,\dots,p_m$ be the distinct prime factors of $\exp(G)$. Let us write $G\cong G_1\times\dots\times G_m$ where each $G_i$ is a $p_i$-group. Then
$$\s(G)\leq \sum_{i=1}^{m}\exp(G_1)\dotsm \exp(G_{i-1})\s(G_i)\leq \exp(G)\left(\frac{\s(G_1)}{\exp(G_1)}+\dots+\frac{\s(G_m)}{\exp(G_m)}\right).$$
\end{lem}
\begin{proof}First, note that $\exp(G)=\exp(G_1)\dotsm \exp(G_m)$. In particular
$$\exp(G_1)\dotsm \exp(G_{i-1})\leq \frac{\exp(G)}{\exp(G_i)}$$
for every $i$, which makes the second inequality true. We prove the first inequality by induction on $m$. If $m=1$, the statement is trivial. If $m>1$, note that we can apply Lemma \ref{lemmaprod} to $H=G_m$ and obtain
$$\s(G)\leq \exp(G_1\times\dots\times G_{m-1})(\s(G_m)-1)+\s(G_1\times\dots\times G_{m-1}).$$
Plugging in $\exp(G_1\times\dots\times G_{m-1})=\exp(G_1)\dotsm\exp(G_{m-1})$ as well as using the induction assumption for $G_1\times\dots\times G_{m-1}$ yields
$$\s(G)\leq \exp(G_1)\dotsm\exp(G_{m-1})\s(G_m)+\sum_{i=1}^{m-1}\exp(G_1)\dotsm \exp(G_{i-1})\s(G_i)=\sum_{i=1}^{m}\exp(G_1)\dotsm \exp(G_{i-1})\s(G_i)$$
as desired.\end{proof}

\begin{lem}\label{lemmaGgeneral} Under the assumptions of Theorem \ref{thm1} we have
$$\s(G)< \exp(G)\left(\frac{\s(\F_{p_1}^{n_1})}{p_1-1}+\dots+\frac{\s(\F_{p_m}^{n_m})}{p_m-1}\right).$$
\end{lem}
\begin{proof} As in Lemma \ref{lemma-product-pgroups}, let us write $G\cong G_1\times\dots\times G_m$ where each $G_i$ is a $p_i$-group. Each $G_i$ can be written as a product of cyclic groups whose orders are powers of $p_i$. Note that the number of factors of each $G_i$ is precisely $n_i$, because together all these factorizations form the unique representation of $G$ as a product of cyclic groups of prime power order. So, by Lemma \ref{lemma-pgroup}, we have
$$\s(G_i)< \frac{\exp(G_i)}{p_i-1}\s(\F_{p_i}^{n_i})$$
for $i=1,\dots,m$. Now the desired inequality follows directly from Lemma \ref{lemma-product-pgroups}.
\end{proof}

\begin{proof}[Proof of Theorem \ref{thm1}]Note that by Theorem \ref{thm4} we have
$$\frac{\s(\F_{p_i}^{n_i})}{p_i-1}\leq \frac{2p_i}{p_i-1}\rt(\F_{p_i}^{n_i})\leq 3\rt(\F_{p_i}^{n_i})$$
for all the odd $p_i$. Since 
$\s(\F_{2}^{n})=2^n+1$ (see \cite[Korollar 1]{harb}) and $\rt(\F_{2}^{n})=2^n$, we also have $\frac{\s(\F_{p_i}^{n_i})}{p_i-1}\leq 3\rt(\F_{p_i}^{n_i})$ if $p_i=2$. Thus, Lemma \ref{lemmaGgeneral} gives
$$\s(G)<\exp(G)\cdot \left(3\rt(\F_{p_1}^{n_1})+\dots+3\rt(\F_{p_m}^{n_m})\right)= 3\exp(G)\cdot (\rt(\F_{p_1}^{n_1})+\dots+\rt(\F_{p_m}^{n_m})),$$
as desired.
\end{proof}

\textit{Acknowledgements.} We would like to thank the anonymous referees for helpful comments, including pointing out the reference \cite{alon}.


\begin{thebibliography}{99}
\bibitem{alon} N. Alon, Subset sums, \textit{J. Number Theory} \textbf{27} (1987), 196--205.

\bibitem{alondub2} N. Alon and M. Dubiner, Zero-sum sets of prescribed size, \textit{Combinatorics, Paul Erd\H{o}s is eighty}, Vol. 1, 33--50, J\'{a}nos Bolyai Math. Soc., Budapest, 1993.

\bibitem{alondub} N. Alon and M. Dubiner, A lattice point problem and additive number theory, \textit{Combinatorica} \textbf{15} (1995), 301--309.

\bibitem{alonshpilkaumans} N. Alon, A. Shpilka, and C. Umans, On sunflowers and matrix multiplication, \textit{Comput. Complexity} \textbf{22} (2013), 219--243.

\bibitem{behrend} F. A. Behrend, On sets of integers which contain no three terms in arithmetical progression, \textit{Proc. Nat. Acad. Sci. U.S.A.} \textbf{32} (1946), 331--332.

\bibitem{blasiaketal} J. Blasiak, T. Church, H. Cohn, J. A. Grochow, E. Naslund, W. F. Sawin, and C. Umans, On cap sets and the group-theoretic approach to matrix multiplication, \textit{Discrete Anal.} 2017, Paper No. 3, 27pp.

\bibitem{chietal} R. Chi, S. Ding, W. Gao, A. Geroldinger, and W. A. Schmid, On zero-sum subsequences of restricted size. IV, \textit{Acta Math. Hungar.} \textbf{107} (2005), 337--344.

\bibitem{chintamanietal}  M. N. Chintamani, B. K. Moriya, W. D. Gao, P. Paul, and R. Thangadurai, New upper bounds for the Davenport and for the Erd\H{o}s-Ginzburg-Ziv constants, \textit{Arch. Math.} \textbf{98} (2012), 133--142.

\bibitem{crootlevpach} E. Croot, V. F. Lev, and P. P. Pach, Progression-free sets in $\Z_4^n$ are exponentially small, \textit{Ann. of Math.} \textbf{185} (2017), 331--337.

\bibitem{edeletal} Y. Edel. C. Elsholtz, A. Geroldinger, S. Kubertin, and L. Rackham, Zero-sum problems in finite abelian groups and affine caps, \textit{Q. J. Math.} \textbf{58} (2007), 159--186.

\bibitem{ellengijs} J. S. Ellenberg and D. Gijswijt, On large subsets of $\mathbb{F}_q^n$ with no three-term arithmetic progression, \textit{Ann. of Math.} \textbf{185} (2017), 339--343.

\bibitem{elsholtz} C. Elsholtz, Lower bounds for multidimensional zero sums, \textit{Combinatorica} \textbf{24} (2004), 351--358.

\bibitem{egz} P. Erd\H{o}s, A. Ginzburg, and A. Ziv, Theorem in the additive number theory, \textit{Bull. Res. Council Israel} \textbf{10F} (1961), 41--43.

\bibitem{foxpham} J. Fox and H. T. Pham, Popular progression differences in vector spaces II, \textit{Discrete Analysis}, to appear.

\bibitem{gaogerold} W. Gao and A. Geroldinger, Zero-sum problems in finite abelian groups: a survey, \textit{Expo. Math.} \textbf{24} (2006), 337--369.

\bibitem{gaoyang} W. D. Gao and Y. X. Yang, Note on a combinatorial constant, \textit{J. Math. Res. Expo.} \textbf{17} (1997), 139--140.

\bibitem{bookgeroldhaltkoch} A. Geroldinger and F. Halter-Koch, Non-Unique Factorizations.\ Algebraic, Combinatorial and Analytic Theory, \textit{Pure and Applied Mathematics}, vol. 278, Chapman \& Hall/CRC, Boca Raton, FL, 2006.

\bibitem{harb} H. Harborth, Ein Extremalproblem f\"ur Gitterpunkte, \textit{J. Reine Angew. Math.}, \textbf{262} (1973), 356--360.

\bibitem{naslund} E. Naslund, Exponential Bounds for the Erd\H{o}s-Ginzburg-Ziv Constant, preprint, 2018, arXiv:1701.04942v2.

\bibitem{reiher} C. Reiher, On Kemnitz' conjecture concerning lattice-points in the plane, \textit{Ramanujan J.} \textbf{13} (2007), 333--337.

\bibitem{salemspencer} R. Salem and D. C. Spencer, On sets of integers which contain no three terms in arithmetical progression, \textit{Proc. Nat. Acad. Sci. U.S.A.} \textbf{28} (1942), 561--563.

\bibitem{tao} T. Tao, A symmetric formulation of the Croot-Lev-Pach-Ellenberg-Gijswijt capset bound, blog post, 2016, \texttt{http://terrytao.wordpress.com/2016/05/18/a}.
\end{thebibliography}
\end{document}